\documentclass[a4paper,10pt]{article}
\usepackage[T1]{fontenc}
\usepackage[cp1250]{inputenc}
\usepackage[english]{babel}
\usepackage{amsmath}
\usepackage{amsfonts}
\usepackage{amssymb}
\usepackage{mathrsfs}
\usepackage{amsthm}
\usepackage{graphicx}

\usepackage{color}
\usepackage{euscript}

\renewcommand{\leq}{\leqslant}
\renewcommand{\geq}{\geqslant}
\newcommand{\E}{\mathbb{E}}
\newcommand{\F}{\mathcal{F}}
\newcommand{\I}{\mathbb{I}}
\newcommand{\BMO}{\text{BMO}}
\def \II#1 {\mathbb{I}_#1}
\renewcommand{\P}{\mathbb{P}}
\newcommand{\R}{\mathbb{R}}
\newcommand{\X}{\mathfrak{X}}
\newcommand{\holders} {H$\ddot{\text{o}}$lder's }

\newtheorem{Le}{Lemma}[section]

\newtheorem{Th}{Theorem}[section]

\numberwithin{equation}{section}

\begin{document}

\title{Limit case of Hardy--Littlewood--Sobolev inequality for martingales \protect\footnote{Supported by the Russian Science Foundation grant N. 19-71-10023.}}
\author{Dmitry Yarcev}
\date{}

\maketitle

\begin{abstract}
We provide a version of the Stein--Weiss inequality for arbitrary martingales.
\end{abstract}

\section{Introduction}
    \par{The classical Hardy--Littlewood--Sobolev plays an important role in analysis. It was initially introduced in \cite{sobolev} as a tool to prove Sobolev's embedding theorem. The inequality says that the Riesz potential (i.e. a Fourier multiplier with the symbol $|\cdot|^\alpha$) is a continuous mapping from $L_p(\R^d)$ to $L_q(\R^d)$, where $\frac{\alpha}{d} = \frac{1}{p}-\frac{1}{q}$ and $1 < p < q < \infty$. The notion of Riesz potential was later transferred to probability theory. Watari provided such a counterpart in \cite{watari} and proved the corresponding version of the Hardy--Littlewood--Sobolev inequality. This idea was generalised further in \cite{naksad}, and the corresponding estimates were proved for regular martingales. Finally, in \cite{us} the authors propose a version of the Riesz potential and prove the corresponding inequality for arbitrary martingales. There are two versions of the Riesz potential in \cite{us}. The one that is studied here provides a more interesting version of the inequality; however, it is not a matringale transform and lacks the semigroup property.}
    \par{We are concerned with the limit case of the Hardy--Littlewood--Sobolev inequality for martingales when $q=\infty$. In this case the Riesz potenial does not necessarily map the martingale space $L_p$ to $L_\infty$; however, it was asked by Prof. Adam Osekowski whether it maps $L_p$ to BMO. This conjecture is proved here along with a few auxiliary inequalities, which might be of some interest in themselves. Notably, one of those inequalities can be seen as the martingale counterpart to the Stein--Weiss inequality (theorem H in \cite{steinweiss}).}
    \par{I am grateful to Prof. Adam Osekowski for asking the question underlying this work and to Prof. Dmitriy Stolyarov for supervising the development of this paper.}

\section{Notation and main result}
    \par{Consider a probability space $\X = (\Omega, \F, \P)$ and a sequence of sigma-algebrae $\F_0 \subset \F_1 \subset ... \subset \F$. We assume that this sequence is a filtration, i.e. for each two distinct points $\omega_1$ and $\omega_2$ in $\Omega$ there is a number $n$ and an event $X \in \F_n$ such that $\omega_1 \in X$ and $\omega_2 \notin X$. We also assume that $\F_0 = \{\varnothing, \Omega\}$ and each subalgebra $\F_n$ is generated by a finite or countable set of atoms. In other words, for each step $n$ there is a finite or countable partition $\Omega = \bigsqcup_{j} a_j$ such that the algebra $\F_n$ is generated by the events $a_j$.}
    \par{For a real parameter $\alpha \in [0, 1)$, define the Riesz potential $I_\alpha$ by the formula}
    \begin{equation} \label{Riesz} I_\alpha[F] = \sum \limits_{n=1}^\infty M_n(E_n-E_{n-1})F, ~F \in L_1(\X) \end{equation}
    \par{There are two families of operators in this formula. The operators $E_n$ perform averaging over algebrae $\F_n$: $E_nf = \E(f \mid \F_n)$. The operators $M_n$ multiply by special functions: $[M_nf](\omega) = b_n^\alpha(\omega)f(\omega)$, where $b_n(\omega) = \P(a_j)$ if the atoms $a_1, a_2, ... $ generate $\F_n$ and $\omega \in a_j$. Note that $(E_n-E_{n-1})F$ is the $n$-th martingale difference of $F$: $(E_n-E_{n-1})F = F_n - F_{n-1}$.}
    \par{Our goal is to prove the limit case of Hardy--Littlewood--Sobolev inequality for martingales. The non-limit case was proved in \cite{us}. We cite it here.}
    \begin{Th}
        For all numbers $p,q$ such that $1 < p < q < \infty,$ the operator $I_\alpha,~\alpha = \frac{1}{p} - \frac{1}{q},$ maps $L_p(\X)$ to $L_q(\X)$ continuously. In other words, there is a constant $C = C(p, q)$ such that for all functions $F \in L_p(\X)$ the series \eqref{Riesz} converges in $L_q$ and the inequality
    \begin{equation} \label{HLS} \|I_\alpha[F]\|_q \leq C \|F\|_p \end{equation}
    holds.
    \end{Th} 
    \par{In this paper, we will be concerned with the limit case of this inequality as $q \to \infty$ . That means that the domain of the Riesz potential is the set $L_p(\X), p > 1,$ and the parameter $\alpha$ is set to $\frac{1}{p}$. In this case, the Riesz potential is not a continuous mapping from $L_p$ to $L_\infty$ (see Section \ref{counterexample}). However, we will prove that it is a continuous mapping between the spaces $L_p$ and BMO.}
    \par{We say that a martingale $\{F_n\}$ with the limit value $F$ belongs to the class BMO = $\BMO(\X, \{\F_n\})$ if for each step $n$ and each atom $a \in \F_n$ the value $\E((F - \E(F_n))^2 \mid a)$ is uniformly bounded. Note that this space depends on the choice of the filtration $\F_n$. The space BMO is a Banach space with the conventional seminorm being}
    $$\|F\|_\BMO = \sup_{n = 0,1,...} \sup_{a \in \F_n} (\E((F - \E(F_n \mid a))^2 \mid a))^\frac{1}{2}.$$
    \begin{Th} \label{main}
        For all numbers $r,~1 < r < \infty,$ the operator $I_\alpha, \alpha = \frac{1}{r},$ maps $L_r(\X)$ to $\BMO(\X, \{\F_n\})$ continuously. In other words, there is a constant $C = C(r)$ such that for all functions $F \in L_r(\X)$ the series \eqref{Riesz} converges in BMO and the inequality
    \begin{equation} \label{HLSBMO} \|I_\alpha[F]\|_{\BMO} \leq C \|F\|_r \end{equation}
    holds.
    \end{Th}
    \par{Our proof will go as follows. We will be dealing with the formally conjugate Riesz potential $I_\alpha'$ instead of the initial one. Its domain is the pre-dual of the BMO space, which is the space $H_1$ of martingales whose maximal functions belong to class $L_1$. Its range is the pre-dual of the $L_r$ space, which is the $L_{r'}$ space ($r'$ is the conjugate exponent of $r$). Thus, the continuity of the operator $I_\alpha'$ is expressed in terms of operators $M_n$, $E_n$ and the maximal function. We refer the reader to \cite{garcia} and \cite{kazamaki} for additional information about martingale BMO and $H_1$ spaces.}
    \par{Consider an arbitrary martingale $F$ of class $H_1$. After a single step, it splits into $s$ conditional martingales, $s \in \{1,2,...,\infty\}$, each defined on their own atoms $a_j, j = 1, ..., s$. For each atom, we obtain four values, named $p_j, x_j, y_j$ and $A_j$. The value $p_j$ is simply the probability $\P(a_j)$, whereas the numbers $x_j, y_j, A_j$ have a more complicated definition in terms of the maximal function and the Riesz potential. We claim that the continuity of the operator $I_\alpha'$ follows from a numerical inequality regarding the four sequences.}
    \begin{Th} \label{numerical}
   Let $\{x_j\}, \{p_j\}, \{A_j\}$ and $\{y_j\}$ be four sequences of $s$ non-negative numbers and $\mu > 0, p > 1$ be real numbers. Suppose that the following conditions hold:
    \begin{equation} \label{cond1} \sum \limits_{j=1}^s p_j = 1;\end{equation}
    \begin{equation} \label{cond2} y_j \geq x_j,~j = 1, ..., s;\end{equation}
    \begin{equation} \label{cond3} (1-p_j)x_j \leq 2 \sum \limits_{k \neq j} x_k,~k = 1, ..., s;\end{equation}
    \begin{equation} \label{cond4} A_j \leq x_j(1-p_j) + p_j^\frac{1}{p}(\sum \limits_{k \neq j} x_k^p)^\frac{1}{p} (1-p_j)^\frac{1}{p'},~j=1,...,s.\end{equation}
    Then, 
    \begin{equation} \label{numineq} \sum \limits_{j=1}^s \mu A_j^p + \sum \limits_{j=1}^s \mu C^\frac{p-1}{p}A_jy_j^{p-1} \leq C((\sum \limits_{j=1}^s y_j)^p - \sum \limits_{j=1}^s y_j^p) \end{equation}
for sufficiently large numbers $C \geq C(p)$.
    \end{Th}
    \par{Theorem \ref{main} may now be proved in the following way. Firstly, we switch from the Riesz potential to its formally conjugate. We also reduce the inequality to the case of simple martingales, i.e. martingales that stop at some moment, as we would like to use induction with respect that moment. Secondly, we introduce the single step model and the auxiliary notation including the four sequences. Then we show that those sequences satisfy the conditions of Theorem \ref{numerical} and prove the numerical inequality itself. Finally, we demonstrate how the continuity of operator $I_\alpha'$ follows from this inequality.}

\section{Conjugate Riesz potential}
    \par{We will work with the formally conjugate Riesz potential $I_\alpha'$ rather than with the operator $I_\alpha$ itself (we use the prime symbol to denote the formally conjugate operator since the asterisk will be associated with the maximal function). When we say ``formally conjugate operator'', we mean that $\E(I_\alpha[f]g) = \E(fI_\alpha'[g])$ for all simple functions $f$ and $g$ (i.e. linear combinations of characteristic functions of atoms). First, we are going to list a few basic properties of the operators $E_n$ and $M_n$, which will help us construct the formally conjugate operator. We omit the proof of the following lemma.}
    \begin{Le} \label{properties}
        The operators $E_n$ and $M_n$ have the following properties.
        \begin{enumerate}
            \item{$E_n' = E_n, M_n' = M_n$;}
            \item{$E_nE_k = E_{\min(n, k)}$;}
            \item{$E_nM_k = M_kE_n$ if $k \leq n$.}
        \end{enumerate}
    \end{Le}
Note that since the operators $E_n$ and $M_n$ do not commute, $I_\alpha \neq I_\alpha'$ in general.
    \par{Now it becomes clear that the formally conjugate Riesz potential can be written as}
    \begin{equation} \label{conjRiesz}  I'_\alpha[F] = \sum \limits_{n=1}^\infty (E_n-E_{n-1})M_nF. \end{equation}
    \par{The domain and range of the formally conjugate operator would be pre-dual spaces to $\BMO(\X, \{\F_n\})$ and $L_r(\X)$, respectively. The former is $H_1(\X, \{\F_n\})$. It is the space of martingales whose maximal function belongs to $L_1(\X)$. The latter is $L_{r'}(\X)$, where $r' = \frac{r}{r-1}$ is the conjugate exponent of $r$.}
    \par{We need to clarify how the $H_1$-norm of a random variable is defined. To compute $\|F\|_{H_1}$, we construct the martingale $F_n = \E(F|\F_n)$. This martingale has the maximal function $F^*$. If this function is summable, its $L_1$-norm $\|F^*\|_1$ is the $H_1$-norm of function $F$.}
    \par{The next theorem describes the continuity of the formally conjugate Riesz potential from $H_1$ to $L_{p}$. We will apply it with $p = r'$.}
    \begin{Th} \label{conjHLSTh}
         For all numbers $p,~1 < p < \infty,$ the operator $I'_\alpha, \alpha = \frac{1}{p'},$ maps $H_1(\X)$ to $L_{p}(\X)$ continuously. In other words, there is a constant $C = C(p)$ such that for all functions $F \in L_p(\X)$ the series \eqref{conjRiesz} converges in $L_p$ and the inequality
    \begin{equation} \label{conjHLS} \|I'_\alpha[F]\|_{p} \leq C \|F^*\|_1 \end{equation}
    holds.
    \end{Th}
    \par{We would like to note that this result is a martingale counterpart to the classical Stein--Weiss inequality (see Theorem H in \cite{steinweiss}).}   

\section{Truncating the series}
    \par{In this section, we will show that it suffices to prove \eqref{conjHLS} only for simple martingales, i.e. martingales that stop at some step. We confirm this claim using identities from Lemma \ref{properties}. Consider a positive integer $N$. Note that if $N \geq n$, then}
    \begin{multline*} E_N(E_n-E_{n-1})M_nF = (E_n-E_{n-1})M_nF = (E_n-E_{n-1})E_nM_nF = \\ = (E_n-E_{n-1})M_nE_nF = (E_n-E_{n-1})M_nF_n.\end{multline*}
    \par{If $N < n$, then}
    $$E_N(E_n-E_{n-1})M_nF = (E_N-E_N)M_nF = 0.$$
    \par{When combined, the two identities become}
    $$E_NI'_\alpha[F] = \sum \limits_{n=1}^N (E_n-E_{n-1})M_nF_n.$$
    \par{In other words, the martingale generated by the random variable $I'_\alpha[F]$ is the martingale}
    $$\{(I'_\alpha[F])_N\} = \{\sum \limits_{n=1}^N (E_n-E_{n-1})M_nF_n\}.$$
    \par{Suppose we have managed to prove \eqref{conjHLS} for simple martingales and are now trying to prove it for an arbitrary function $F$. Since $(I'_\alpha)_N$ only depends on $F_0, F_1, ..., F_N$, for each $N$ we have the bound}
    $$\|(I'_\alpha)_N\|_p \leq C \|F_N^*\|_1 \leq C \|F^*\|_1.$$
    \par{Thus, the martingale $\{(I'_\alpha)_N\}$ belongs to $L_p$ and has $L_p$-norm less than or equal to $C\|F^*\|_1$. Since $p > 1$, by Doob's convergence theorem the series \eqref{conjRiesz} converges in $L_p$ to a random variable and Theorem \ref{conjHLSTh} holds.}
    \par{From now on we will be concerned with the simple martingale version of \eqref{conjHLS}, which states}
    $$\|(I'_\alpha)_N\|_p \leq C \|F_N^*\|_1.$$

\section{Single step model}
    \par{Truncating the series in the previous section allows us to use induction in $N$. The induction basis for $N=0$ is trivial. To prove the induction step, we introduce new notation.} 
    \par{Consider the first step of the martingale $F$ with mean value $\bar{f} = F_0$. The corresponding algebra $\F_1$ is generated by the atoms $a_j, j = 1,...,s$ (with, possibly, $s = \infty$). For each atom $a_j$ of positive probability $p_j$ there is a conditional subspace $\X_j = (a_j, \F \cap 2^{a_j}, \frac{1}{p_j}\P \mid_{a_j})$. The restrictions of random variables $F_1, F_2, ...$ on the conditional subspaces form conditional martingales, which we label as $G^j, j = 1,..,s$. We call the mean values of those martingales (i.e. values of the variable $F_1$ on different atoms) $f_j$.}
   \par{Now recall the definition of the conjugate Riesz potential $I_\alpha'[F]$ as the sum \eqref{conjRiesz}. We define a non-negative random variable $A$ so that the first summand is $\pm A$, and a non-negative random variable $B$ such that the remaining sum is $\pm B$. The $L_1$-norm of the maximal function $F^*$ is a sum of integrals over different atoms. We denote the integral over the atom $a_j$ by $y_j$. Using the same concept for the single step maximal function, we obtain smaller values $x_j$. The Values $A_j$ are obtained in a similar manner, but instead of identities $\|F^*\|_1 = \sum y_j$ and $\|F_1^*\|_1 = \sum x_j$ we have $\|A\|_p^p = \sum A_j^p$.}
   \par{All the new notations along with their precise definitions are summarised in the table below.}

\begin{table}[h]
\begin{tabular}{|c|c|c|}
\hline
Symbol & Formula & Meaning \\
\hline 
$a_j, j = 1, ..., s$ & N/A & Atoms generating the algebra $\F_1$\\
\hline
$\X_j, j = 1, ..., s$ & $\X_j = (a_j, \F \cap 2^{a_j}, \frac{1}{\P(a_j)}\P \mid_{a_j})$ & Conditional subspaces of $\X$\\
\hline
$p_j, j = 1, ..., s$ & $p_j = \P(a_j)$ & Probabilities of the atoms $a_j$\\
\hline
$\{(G^j)_n\}, j = 1, ..., s$ &  $\{(G^j)_n\} = \{F_{n+1}\mid_{a_j}\}$ & Conditional martingales \\
\hline
$f_j, j = 1, ..., s$ & $f_j = F_1|_{a_j}$ & Values of martingale step $F_1$ on each atom\\
\hline
\rule{0pt}{11pt} $\bar{f}$ & $\bar{f} = F_0$ & The weighted average of values $f_j$\\
\hline
$A$ & $A = |(E_1-E_0)M_1F_1|$ & The first summand of the conjugate Riesz potential\\
\hline
$B$ & $B = |\sum \limits_{n=2}^{N+1} (E_n-E_{n-1})M_nF_n|$ & The remaining summands\\
\hline
$x_j, j = 1, ..., s$ & $x_j = p_j\max(|f|,|f_j|)$ & Fragments of single step maximal function $F_1^*$ \\
\hline
$y_j, j = 1, ..., s$ & $y_j = \int \limits_{a_j} F^*d\P$ & Fragments of the overall maximal function $F^*$ \\
\hline
\rule{0pt}{15pt} $A_j, j = 1, ..., s$ & $A_j = p_j^\frac{1}{p} A|_{a_j}d\P$ & Pseudofragments of the $L_p$-norm of summand $A$\\
\hline
\end{tabular}
\end{table}

\section{Verification of the conditions}
   \par{Before proving Theorem \ref{numerical}, we check that it is applicable to the four sequences $p_j, x_j, y_j, a_j$ constructed above. Conditions \eqref{cond1} and \eqref{cond2} are evident by construction. Now let us verify the condition \eqref{cond3}, which we will refer to as ``the non-singularity condition'' from now on.}
   \begin{Le} The inequality $(1-p_j)x_j \leq 2 \sum \limits_{k \neq j} x_k$ holds for all j = 1,...,s. \end{Le}
   \begin{proof}
       \par{If $x_j \leq 2 |\bar{f}|$, then}
       $$(1-p_j)x_j \leq 2(1-p_j)|\bar{f}| = 2\sum \limits_{k \neq j} p_k|\bar{f}| \leq 2 \sum \limits_{k \neq j} x_k.$$
       \par{If $x_j > 2 |\bar{f}|$, then, obviously, $x_j = p_j|f_j|$, so}
       $$x_j = |\bar{f} - \sum \limits_{k \neq s} p_kf_k| \leq |\bar{f}| + \sum \limits_{k \neq j} p_k|f_k| \leq \frac{x_j}{2} + \sum \limits_{k \neq j} x_k,$$
       $$\frac{x_j}{2} \leq \sum \limits_{k \neq j} x_k,~(1-p_j)x_j \leq x_j \leq 2 \sum \limits_{k \neq j} x_k.$$
   \end{proof}
   \par{The last condition to verify is \eqref{cond4}.}
    \begin{Le} The inequality $A_j \leq x_j(1-p_j) + p_j^\frac{1}{p}(\sum \limits_{k \neq j} x_k^p)^\frac{1}{p} (1-p_j)^\frac{1}{p'}$ holds for all j = 1, ..., s (recall that $p'$ is the conjugate exponent $\frac{p}{p-1}$). \end{Le}
    \begin{proof} Recall that
    $$F_1 = \sum \limits_{j=1}^s f_j \I_{a_j},~p_j = \P(a_j),~\bar{f} = \sum \limits_{k=1}^s p_kf_k,~x_j = \|F_1^*\I_{a_j}\|_1 = p_j\max(|f_j|, |\bar{f}|),$$
    and so
    $$M_1F_1 = E_1M_1F_1 = \sum \limits_{j=1}^s f_j p_j^\frac{1}{p'} \I_{a_j},~E_0M_1F_1 = \sum \limits_{j=1}^s f_j p_j^{1+\frac{1}{p'}},$$
    $$A|_{a_j} = |E_0M_1F_1 - M_1F_1| |_{a_j} = |f_j p_j^\frac{1}{p'}(1-p_j) - \sum \limits_{k \neq j} f_k p_k^{1 + \frac{1}{p'}}|,$$
    \begin{multline*} \label{abound} A_j = p_j^\frac{1}{p}A|_{a_j} = |f_j p_j(1-p_j) - \sum \limits_{k \neq j} f_k p_k^{1 + \frac{1}{p'}}p_j^\frac{1}{p}| \leq x_j(1-p_j) + \sum \limits_{k \neq j} x_k p_k^{\frac{1}{p'}}p_j^\frac{1}{p} \leq \\ \leq x_j(1-p_j) + p_j^\frac{1}{p}(\sum \limits_{k \neq j} (f_k p_k)^p)^\frac{1}{p}(\sum \limits_{k \neq j} p_k^\frac{p'}{p'})^\frac{1}{p'} \leq x_j(1-p_j) + p_j^\frac{1}{p}(\sum \limits_{k \neq j} x_k^p)^\frac{1}{p} (1-p_j)^\frac{1}{p'}. \end{multline*} 
   \end{proof}

\section{Proof of the numerical theorem}
    \par{Having obtained the four conditions, we now fully detach ourselves from the martingale nature of the problem and prove the numerical inequality \eqref{numineq}. A tool used to prove this theorem is the non-singularity lemma. It is a quantification of the following statement: if neither of the numbers $x_j$ is significantly larger than the others, then the sum $\sum x_j^p$ cannot constitute a large portion of $(\sum x_j)^p$.}
    \begin{Le} \label{nonsinglemma}
        If non-negative numbers $x_j$ satisfy the non-singularity condition \eqref{cond3} with probabilities $p_j$ and the number $p_1$ is the greatest of them, then
        $$\big(\sum \limits_{j=1}^s x_j\big)^p - \sum \limits_{j=1}^s x_j^p \gtrsim (1-p_1)x_1^p + \big(\sum \limits_{j > 1} x_j\big)^p.$$
    \end{Le}
    \begin{proof}
        \par{Since the inequality is homogenous, we may assume $\sum \limits_{j=1}^s x_j = 1$. Note that}
        $$1 - \sum \limits_{j=1}^s x_j^p \geq 1 - (\sum \limits_{j=1}^s x_j) \max \limits_{j=1}^s x_j^{p-1} = 1 - \max \limits_{j=1}^s x_j^{p-1}.$$
        \par{If neither of the numbers $x_i$ is greater than $\frac{4}{5}$, then the final difference is at least $1 - (\frac{4}{5})^{p-1} \gtrsim 1 \geq (1-p_1)x_1^p + (\sum \limits_{j > 1} x_j)^p$. If there is a number $x_j > \frac{4}{5}$, then it must be $x_1$ since from the non-singularity condition $x_j \leq \frac{2}{3-p_j}$ and all probabilities $p_j$ except $p_1$ are at most $\frac{1}{2}$. If $x_1 > \frac{4}{5}$, denote $\sum \limits_{j>1}x_j$ by $x$. Then}
    \begin{multline*}(\sum \limits_{j=1}^s x_j)^p - \sum \limits_{j=1}^s x_j^p \geq (x_1+x)^p - x_1^p - x^p = \int \limits_{x_1}^{x_1+x}pt^{p-1}dt - x^p \geq pxx_1^{p-1} - x^p \geq \\ \geq (p-1)xx_1^{p-1} \gtrsim 3xx_1^{p-1} \geq x^p + (1-p_1)x_1^p\end{multline*}
    \par{(the final step uses the non-singularity condition $(1-p_1)x_1 \leq 2x$ and the inequality $x_1>x$).}
    \end{proof}    
    \par{Our next goal is to reduce \eqref{numineq} to the non-singular case when $y_k(1-p_k) \leq 2 \sum \limits_{j \neq k} y_j$  for each $k$. Consider the difference between the right hand and left hand parts of \eqref{numineq} as a function of $y_k$. Its derivative is equal to}
    $$C p\big[(\sum \limits_{j=1}^s y_j)^{p-1}-y_k^{p-1}\big] - \mu(p-1)C^\frac{p-1}{p}A_ky_k^{p-2}.$$
    \par{Denote $\frac{\sum \limits_{j \neq k} y_j}{y_k}$ by $\theta$. Then the derivative is equal to}
    $$C p y_k^{p-1}\big[(1+\theta)^{p-1} - 1 - \frac{\mu(p-1)}{pC^\frac{1}{p}} \frac{A_k}{y_k}\big].$$
    \par{The bound \eqref{cond4} along with the non-singularity condition implies}
    $$A_k \leq 3 \sum \limits_{j \neq k} x_j \leq 3 \sum \limits_{j \neq k} y_j,$$
    {so the derivative is at least}
    $$C p y_k^{p-1}((1+\theta)^{p-1} - 1 - \frac{3\mu(p-1)}{pC^\frac{1}{p}}\theta).$$
    \par{When $\theta < 2$ and the constant $C$ is so large that $(p-1)\max(1, 3^{p-2}) \geq \frac{3\mu(p-1)}{pC^\frac{1}{p}}$, we can see that}
    $$(1+\theta)^{p-1} - 1 - \frac{3(p-1)}{C^\frac{1}{p}}\theta = \int \limits_{0}^\theta (p-1)(1+t)^{p-2} - \frac{3\mu(p-1)}{pC^\frac{1}{p}} dt \geq 0,$$
    {so the derivative is non-negative. That means we may reduce the number $y_k$ until either $y_k = x_k$ or $\theta = 2$, and the resulting inequality will imply the initial one. In the former case $y_k(1 - p_k) = x_k(1 - p_k) \leq 2 \sum \limits_{j \neq k} x_j \leq 2 \sum \limits_{j \neq k} y_j$, and in the latter $y_k(1-p_k) \leq y_k = 2 \sum \limits_{j \neq k} y_j$. Note that if the number $y_k$ was reduced this way, then $y_k \geq y_j$ for all $j$ both before and after the reduction. Since $y_k(1-p_k) \leq 2 \sum \limits_{j \neq k} y_j$ by construction and $y_{k'}(1-p_{k'}) \leq 2 \sum \limits_{j \neq {k'}} y_j$ follows from $y_{k'} \leq y_k$ for all other $k'$, the non-singularity condition holds for the new numbers $y_j$.}
    \par{Now we may increase each number $x_j$ to its respective $y_j$. The inequality \eqref{numineq} is not affected by this, and the upper bound \eqref{cond4} only increases. All that remains is to prove the inequality when $x_j = y_j$.}
    \par{The inequality \eqref{numineq} is equivalent to two separate inequalities.}
    \begin{Le}
        The following inequalities hold:
            \begin{equation} \label{ineq1} \sum \limits_{j=1}^s A_j^p \lesssim (\sum \limits_{j=1}^s y_j)^p - \sum \limits_{j=1}^s y_j^p; \end{equation}
            \begin{equation} \label{ineq2} \sum \limits_{j = 1}^s A_jy_j^{p-1} \lesssim (\sum \limits_{j=1}^s y_j)^p - \sum \limits_{j=1}^s y_j^p. \end{equation}
    \end{Le}
    \begin{proof} \par{If \eqref{ineq1} holds with constant $C_1$ and \eqref{ineq2} holds with constant $C_2$, then \eqref{numineq} holds for constants $C$ such that $\mu C_1 + \mu C^\frac{p-1}{p}C_2 \leq C$. The proof for each of the inequalities is rather straightforward.}
    \par{We start with \eqref{ineq1}. Note that $(\frac{x+y}{2})^p \leq \frac{x^p+y^p}{2}$ from Jensen's inequality for the function $|x|^p$, so, $(x+y)^p \lesssim x^p + y^p$. Then, \eqref{cond4} implies}
    $$\sum \limits_{j=1}^s A_j^p \lesssim \sum \limits_{j=1}^s y_j^p(1-y_j)^p + \sum \limits_{k \neq j} p_jy_k^p = \sum \limits_{j=1}^s y_j^p((1-p_j)^p + (1-p_j)) \lesssim \sum \limits_{j=1}^s y_j^p(1-p_j) \lesssim (\sum \limits_{j=1}^s y_j)^p - \sum \limits_{j=1}^s y_j^p$$
    \par{(the final step uses Lemma \ref{nonsinglemma}). Moving to \eqref{ineq2}, we immediately see that}
    $$\sum \limits_{j = 1}^s A_jy_j^{p-1} \leq \sum \limits_{j = 1}^s (y_j^p(1-p_j) + (y_j^p(1-p_j))^\frac{1}{p'} (\sum \limits_{k \neq j} p_jy_k^p)^\frac{1}{p}).$$
    \par{Applying \holders inequality to the second summand yields}
    \begin{multline*} \sum \limits_{j = 1}^s (y_j^p(1-p_j))^\frac{1}{p'}(\sum \limits_{k \neq j} p_jy_k^p)^\frac{1}{p} \leq (\sum \limits_{j = 1}^s (y_j^p(1-p_j)))^\frac{1}{p'} (\sum \limits_{j = 1}^s \sum \limits_{k \neq j} p_jy_k^p)^\frac{1}{p} = \\ = (\sum \limits_{j = 1}^s (y_j^p(1-p_j)))^\frac{1}{p'}(\sum \limits_{j = 1}^s(y_j^p(1-p_j)))^\frac{1}{p} = \sum \limits_{j = 1}^s(y_j^p(1-p_j)). \end{multline*}
    \par{Overall, the sum is at most $2\sum \limits_{j = 1}^s(y_j^p(1-p_j)) \lesssim (\sum \limits_{j=1}^s y_j)^p - \sum \limits_{j=1}^s y_j^p$ by Lemma \ref{nonsinglemma}. This concludes the proof.}\end{proof}

\section{Obtaining the continuity of the conjugate Riesz potential}
    \par{Recall that the functions $A$ and $B$ were chosen in such a way that $I_\alpha'[F] = \pm A \pm B$, and, by definition of the values $y_j$, $\|F^*\|_1 = \sum \limits_{j=1}^s y_j$. Thus, \eqref{conjHLS} can be written as}
    \begin{equation} \int \limits_{\Omega} |\pm A \pm B|^p d\P \leq C^p(\sum \limits_{j=1}^s y_j)^p \end{equation}
    {for some constant $C$, and we will prove its stronger form}
    \begin{equation} \label{indHLS} \int \limits_{\Omega} (A+B)^p d\P \leq C(\sum \limits_{j=1}^s y_j)^p \end{equation}
for some other constant, which we also call $C$.
    \par{Our first step in proving this inequality is to obtain an upper bound for the $L_p$-norm of the function $B$ using induction hypothesis. Note that}
    $$ ((E_{n+1}-E_n)M_{n+1}F_{n+1})|_{a_j} = (E_n - E_{n-1})(p_j^{\frac{1}{p'}}M_n)G_n^j.$$
    \par{Here we slightly abuse the notation: on the left hand side, the operators $E_n$ and $M_{n+1}$ are operators of averaging and multiplication on $\X$, whereas the operators on the right hand side act on $\X_j$. Using the induction hypothesis, we see}
    $$B |_{a_j} = p_j^\frac{1}{p'} I'_\alpha[G^j],$$
    \begin{equation} \label{bbound} \int \limits_{a_j} |B|^p d\P = p_j  \int \limits_{a_j} |B|^p d\P_j = p_j^{1 + \frac{p}{p'}} \int \limits_{a_j} |I'_\alpha[G^j]|^p d\P_j \leq Cp_j^p \|F^*|_{a_j}\|^p_1 = Cy_j^p. \end{equation}
    \par{To conclude the proof, we use the properties of the power functions to split the power $(A+B)^p$ into three parts. Our proof will be slightly different depending on whether $1 < p \leq 2$ or $p > 2$. If $1 < p \leq 2$, then $(x+y)^{p-1} \leq x^{p-1}+y^{p-1}$ for any positive numbers $x$ and $y$. If $p > 2$, this inequality does not hold, but the function $|\cdot|^{p-1}$ is convex. Thus, using Jensen's inequality, we note that $(\frac{x+y}{2})^{p-1} \leq \frac{x^{p-1}+y^{p-1}}{2}$, or $(x+y)^{p-1} \leq 2^{p-2} (x^{p-1}+y^{p-1})$. The two cases can be combined into}
    $$(x+y)^p \leq \max(1, 2^{p-2})(x^{p-1}+y^{p-1}).$$
    \par{Using this inequality and the convexity of the function $|\cdot|^p$ (which is true in both cases), we get the bound
    $$(A+B)^p = B^p + \int \limits_{B}^{A+B} pt^{p-1}dt \leq B^p + Ap(A+B)^{p-1} \leq B^p + \mu(A^p + AB^{p-1}),$$
    \par{where $\mu = p\max(1, 2^{p-2})$. For each atom $a_j$ this means}
    $$\int \limits_{a_j} (A+B)^p d\P \leq \int \limits_{a_j} B^pd\P + \mu A_j^p + \mu A|_{a_j} \int \limits_{a_j} B^{p-1}d\P.$$ 
    \par{Applying \holders inequality to the last integral, we get}
    $$ \int \limits_{a_j} (A+B)^p d\P \leq \int \limits_{a_j} B^pd\P + \mu A_j^p + \mu A|_{a_j} \|\I_{a_j}\|_p \|B^{p-1}\I_{a_j}\|_{p'} = \int \limits_{a_j} B^pd\P + \mu A_j^p + \mu A_j(\int \limits_{a_j} B^pd\P)^\frac{p-1}{p}.$$
    \par{Finally, substitution of $\int \limits_{a_j} B^pd\P$ with the bound from \eqref{bbound} yields}
    \begin{equation} \label{threesummands} \int \limits_{a_j} (A+B)^p d\P \leq Cy_j^p + \mu A_j^p + C^\frac{p-1}{p} \mu A_jy_j^{p-1}. 
\end{equation}
    \par{Substituting the left-hand part in \eqref{indHLS} with a sum of these estimates and moving $\sum \limits_{j=1}^s y_j^p$ to the other side, we come to the inequality}
    $$\sum \limits_{j=1}^s \mu A_j^p + \sum \limits_{j=1}^s \mu C^\frac{p-1}{p}A_jy_j^{p-1} \leq C((\sum \limits_{j=1}^s y_j)^p - \sum \limits_{j=1}^s y_j^p). $$
   \par{Since the four sequences $p_j, x_j, y_j, A_j$ satisfy the four conditions of Theorem \ref{numerical}, this inequality holds true, from which the continuity of the conjugate Riesz potential follows, completing the proof of Theorem \ref{main}.}
	
\section{ \label{counterexample} Lack of continuity to $L_\infty$}
	\par{Since the operator $I_\alpha$ is a continuous mapping from $L_p$ to $L_q$ when $\alpha = \frac{1}{p}-\frac{1}{q}$ and $1 < p < q < \infty$, it is somewhat natural to expect that it is also a continuous mapping from $L_{p'}$ to $L_\infty$ when $\alpha = \frac{1}{p'}$. However, we demonstrate that this is not true for any filtration. To prove that, we fix the filtration and show that for such filtration the partial sums of the series \eqref{conjRiesz} are not uniformly bounded as operators from $L_1$ to $L_{p}$. Thus, their formal conjugates do not converge in norm to a bounded operator from $L_{p'}$ to $L_\infty$.}
	\par{Consider a filtration $\{\F_n\}_n$ on the space $\X = (\Omega, \F, \P)$ and construct a sequence of atoms $a_n \in \F_n$ with the following properties:}
    \begin{enumerate}
        \item{$a_0 = \Omega$;}
        \item{for each $n$ $a_{n+1} \subset a_n$ and $\P(a_n) > 0$;}
        \item{Either $\P(a_{n+1}) = \P(a_n)$ or $\P(a_{n+1}) \leq \frac{1}{2} \P(a_n)$.}
    \end{enumerate}
    \par{Given a fixed atom $a_n$, we select the atom $a_{n+1}$ the following way. If $a_n$ is partitioned into smaller sub-atoms during step $n+1$, then we choose $a_{n+1}$ as the second greatest in probability among those sub-atoms. If $a_n$ is not partitioned, then we have to choose $a_{n+1} = a_n$. Since the probability space $\X$ has no atoms and $\{\F_n\}$ is a filtration, we have $\P(a_n) \to 0$ as $n \to \infty$. Denote $d_n = \P(a_n)$.}
    \par{Consider the martingale $F_n = d_n^{-1} \I_{a_n}$. Clearly, the martingale has the $L_1$-norm of 1. We also have $E_n M_n F_n = M_n F_n = d_n^{-\frac{1}{p}}\I_{a_n}$ and $E_{n-1}M_nF_n = \frac{d_n^\frac{1}{p'}}{d_{n-1}}\I_{a_{n-1}}$. Now we calculate the conjugate Riesz potential of the martingale $F$:}
    \begin{multline*} I'_\alpha[F] = \sum \limits_{n=1}^\infty (E_n-E_{n-1})M_n F_n = -F_0 + \sum \limits_{n=0}^\infty E_n(M_n - M_{n+1})F_n = \\ = -F_0 + \sum \limits_{n=0}^\infty E_n(d_n^{\frac{1}{p'}} - d_{n+1}^{\frac{1}{p'}})F_n =  \sum \limits_{n=0}^\infty (d_n^{-\frac{1}{p}} - \frac{d_{n+1}^{\frac{1}{p'}}}{d_n})\I_{a_n} - 1.\end{multline*}
    \par{Note that if $d_n = d_{n+1}$, then the corresponding summand $(d_n^{-\frac{1}{p}} - \frac{d_{n+1}^{\frac{1}{p'}}}{d_n}) \I_{a_n}$ is zero, and if $d_{n+1} \leq \frac{1}{2}d_n$, then}
    $$(d_n^{-\frac{1}{p}} - \frac{d_{n+1}^{\frac{1}{p'}}}{d_n}) \geq (1-2^{-\frac{1}{p'}}) d_n^{-\frac{1}{p}}.$$ 
    \par{Remove duplicates from the sequence $\{a_n\}$ to get an infinite subsequence $\{a_{n_k}\}$. Using the new notation, we get a pointwise estimate}
    $$I'_\alpha[F] \geq (1 - 2 ^{-\frac{1}{p'}}) \sum \limits_{k=0}^\infty d_{n_k}^{-\frac{1}{p}} \I_{a_{n_k}} - 1,$$
    $$I'_\alpha[F]|_{a_{n_k} \setminus a_{n_{k+1}} } \geq -1 + (1 - 2 ^{-\frac{1}{p'}}) \sum \limits_{l=0}^k d_{n_l}^{-\frac{1}{p}} \gtrsim \sum \limits_{l=0}^k d_{n_l}^{-\frac{1}{p}} \geq d_{n_k}^{-\frac{1}{p}}$$
for sufficiently large $k \geq K(\X, p)$ (i.e. such that $(1 - 2 ^{-\frac{1}{p'}})d_{n_k}^{-\frac{1}{p}} > 2$). Then 
   $$\|I'_\alpha[F]\|_p^p \gtrsim \sum \limits_{k = K(\X, p)}^\infty \frac{d_{n_k}-d_{n_{k+1}}}{d_{n_k}} \geq \sum \limits_{k = K(\X, p)}^\infty \frac{1}{2}.$$
    \par{Thus, the partial sums cannot be uniformly bounded operators from $L_1$ to $L_p$, and the formally conjugate operators cannot converge to a bounded operator from $L_{p'}$ to $L_\infty$.}

\end{document}